\documentclass[9pt]{article} 

\usepackage{amsmath, amsthm, amssymb,geometry} 
\usepackage[applemac]{inputenc}
\usepackage{relsize}
\usepackage[all]{xy} 
\usepackage{textcomp} 
\usepackage{color}
\usepackage{mathrsfs}
\usepackage{leftidx}
\usepackage{bm}
\DeclareMathOperator{\tr}{Tr}

\DeclareMathOperator{\Rm}{Rm}

\newcommand{\R}{\mathbb R}

\newcommand{\diff}{\mathrm{d}}
\newcommand{\dd}{\mathrm{d}}

\renewcommand{\mod}{\text{\rm mod}\,}

\renewcommand{\tilde}{\widetilde}

\theoremstyle{plain}
\newtheorem{theorem}{Theorem}
\newtheorem{proposition}[theorem]{Proposition}
\newtheorem{lemma}[theorem]{Lemma}

\theoremstyle{definition}
\newtheorem{definition}[theorem]{Definition}
\newtheorem{remark}[theorem]{Remark}

\theoremstyle{plain}
\newtheorem*{theorem*}{Theorem}
\newtheorem*{proposition*}{Proposition}
\newtheorem*{lemma*}{Lemma}
\newtheorem*{corollary*}{Corollary}
\newtheorem*{conjecture*}{Conjecture}
\theoremstyle{definition}
\newtheorem*{definition*}{Definition}
\newtheorem*{remark*}{Remark}
\newtheorem*{remarks*}{Remarks}

\makeatletter
\def\blfootnote{\xdef\@thefnmark{}\@footnotetext}
\makeatother

\numberwithin{theorem}{section}
\numberwithin{equation}{section}
\title{Cohomogeneity-one $G_2$-Laplacian flow on 7-torus}
\date{}                                           % Activate to display a given date or no date
\author{Hongnian Huang, Yuanqi Wang, Chengjian Yao}

\begin{document}
	\maketitle
	\begin{abstract}We prove the hypersymplectic flow of simple type on standard torus $\mathbb{T}^4$ exists for all time and converges to 
		the standard flat structure modulo diffeomorphisms. This result in particular gives the first example of a cohomogeneity-one $G_2$-Laplacian flow on a compact $7$-manifold 
		which exists for all time and converges to a torsion-free $G_2$ structure modulo diffeomorphisms.\footnote{YW is supported by Simons Collaboration in Geometry, Analysis, and Physics.  CY is supported by FNRS grant MIS.F.4522.15 and the ERC consolidator grant 646649  ``SymplecticEinstein''. This material is based upon work supported by the National Science Foundation under Grant No. DMS-1440140 while HH and CY were in residence at the	Mathematical Sciences Research Institute in Berkeley, California, during the Spring 2016 semester.}
	\end{abstract}

	%\tableofcontents
	\section{Introduction}
	Let $\underline\omega=(\omega_1, \omega_2, \omega_3)$ be a triple of $2$-forms on a differentiable $4$-manifold $X^4$, it is called a
	\emph{definite triple} if there exists a nowhere vanishing $4$-form $\mu$ on $X$ such that the matrix $\left( \frac{\omega_i\wedge\omega_j}{2\mu}\right)$ is
	positive definite everywhere on $M$.  
	If moreover if $\underline\omega$ is closed, it is called a \emph{hypersymplectic structure}. 
	
	Donaldson raised an open question in \cite{D1}: does $X^4$ admit a hyperK\"ahler structure if there is a hypersymplectic structure on $X$?\; The speculated answer for compact $X$ is 
	``\emph{YES}'', and Donaldson described a general constraint PDE of elliptic type to attack it. In \cite{FY}, a geometric flow is introduced to deform each symplectic
	form in its cohomology class simultaneously and the stationary solution of the flow is a hyperK\"ahler triple.  
	A significant fact is that this flow is the ``gradient flow" of some ``volume functional'', which is bounded from above by topological data and whose critical point (if it exists) is exactly the hyperK\"ahler structure in the same cohomology class (see \cite{FY}).  Let us recall the definition of this flow. \\
	
	Each hypersymplectic structure $\underline\omega$ on $X$ canonically determines a conformal structure $\mathscr{C}_{\underline\omega}$ on it, where 
	\[
	\Lambda^+_{\mathscr{C}_{\underline\omega}}=\text{Span}\{\omega_1, \omega_2, \omega_3\}
	\]
	There is a corresponding Riemannian metric $g\in \mathscr{C}_{\underline\omega}$ which is determined by the following formula:
	
	\begin{equation}\label{metric-from-definite-triple}
		g(u, v)\dd\text{vol}_g = \frac{1}{6}\epsilon_{ijk}\iota_u \omega_i\wedge\iota_v\omega_j\wedge\omega_k, \;\; \forall u, v\in TX
	\end{equation}
	We use $\mu$ to denote the volume form $\dd\text{vol}_g$ of this corresponding metric $g$ for simplicity of notation. Write 
	\begin{equation}
		Q=\big(Q_{ij}\big)=\Big(\frac{\omega_i\wedge\omega_j}{2\mu}\Big) = \Big(\frac{1}{2}\langle\omega_i, \omega_j\rangle\Big)
	\end{equation}
	for the matrix of inner products of the $\omega_i$'s, then it follows easily that $\det Q=1$ (cf. \cite[Lemma 2.4]{FY}). \\
	
	If $Q$ is a constant matrix on some open set, then $g$ is hyperK\"ahler on this open set, and $\underline\omega$ is the corresponding triple of $2$-forms of this corresponding 
	hyperK\"ahler structure up to a constant linear transformation.  We define three endomorphisms of $\Lambda^1$ by the formula:
	$$E_i : \Lambda^1 \rightarrow \Lambda^1$$
	$$\alpha \mapsto -*(\alpha\wedge\omega_i)$$
	If $\omega$ is a self-dual $2$-form with $|\omega|^2=2$, then this endomorphism is the corresponding almost complex structure $J$. We can then define three \emph{torsion $1$-forms}
	\begin{equation}\label{torsion-one-form}
		\tau_i= - \sum_{j=1}^3 E_j(\dd QQ^{-1})_{ij}, \;\; i=1, 2, 3
	\end{equation}
	
	The \emph{hypersymplectic flow} is the following system of PDE:
	
	\begin{equation}\label{definite-triple-flow}
		\partial_t \omega_i= \dd \tau_i, \;\; i=1, 2, 3
	\end{equation}
	
	This flow is intimately linked with \emph{$G_2$-Laplacian flow }, which was introduced by Bryant \cite{B} and Hitchin \cite{H} to study the existence of Riemannian metrics with holonomy group contained in $G_2$ on $7$-manifold.  On $M^7=X^4\times \mathbb{T}^3$ whose angular coordinates on $\mathbb{T}^3$ are denoted by $t_1, t_2, t_3\in [0, 2\pi)$, the $3$-form 
	\begin{equation} 
		\phi=\dd t^1\wedge\dd t^2\wedge\dd t^3 - \dd t^1\wedge\omega_1 - \dd t^2\wedge \omega_2 - \dd t^3\wedge\omega_3
	\end{equation}
	defines a \emph{closed $G_2$ structure} if and only if $\underline\omega$ is a \emph{hypersymplectic structure}.  The $2$-form 
	\begin{equation}\label{intrinsic-torsion}
		\bm\tau =  \sum_{i=1}^3 \diff t_i \wedge \tau_i
	\end{equation}
	is the intrinsic torsion form of $\phi$. Since $\Delta_\phi\phi=\dd \bm\tau$, the solution to the \emph{hypersymplectic flow} on $X$ gives a solution to
	\begin{equation}\label{G2-Laplacian-flow}
		\partial_t \phi= \Delta_\phi\phi
	\end{equation}
	i.e. the \emph{$G_2$-Laplacian flow} on $M$.  There are several important results regarding the $G_2$-Laplacian flow:  the short time existence was proved by \cite{BX}, a Shi-type estimate and compactness result
	were proved by \cite{LW1}, the dynamical stability was proved in \cite{LW2}. 
	% and homogeneous $G_2$-Laplacian flow is studied by \cite{L, FFM, N, L2}. 
	The problem of long time existence has seen a lot of advances: the case of left-invariant closed $G_2$ structures on nilpotent Lie groups was obtained by \cite{FFM}; the case of homogeneous $G_2$-Laplacian flows on solvable Lie groups with a codimension-one Abelian normal subgroup was obtained by \cite{L1}, and there are lots of studies of the corresponding homogeneous Laplacian solitons in \cite{FFM}, \cite{L2}, \cite{N}. In general, assuming $|\Delta_\phi\phi|_\phi$ is uniformly bounded, the long time existence of $G_2$-Laplacian flow was obtained by \cite{LW1}. For the hypersymplectic flow on a compact $4$-manifold, the long time existence assuming bounded torsion was obtained in \cite{FY}.  There is also an interesting reduction of a warped $G_2$-Laplacian flow on $Y^6\times S^1$ to a coupled flow of the $SU(3)$-structure and the warped function on $Y^6$ in \cite{FR}.  \\
	
	The main result of this short article is \emph{long time existence} (Theorem \ref{long-time-existence}) and {convergence}  (Theorem \ref{convergence}) result for \emph{hypersymplectic flow} for \emph{hypersymplectic structures of simple type} (defined in next section) on the $4$-torus $\mathbb{T}^4$. These structures have $\mathbb{T}^3$ symmetry. It is worthing noting that the local boundary value problem for general hypersymplectic structures and torsion free hypersymplectic structures with $S^1$ symmetry was recently investigated in Donaldson \cite{D2} by a generalized Gibbons-Hawking construction.  In contrast to the homogeneous situations in \cite{FFM} and \cite{L1}, \cite{L2}, \cite{N}, the structures considered in this paper are of cohomogeneity-one. To the best of our knowledge, this result provides the first source of cohomogeneity-one $G_2$-Laplacian flows on a compact $7$-manifold which exist globally and converge (modulo diffeomorphisms).  It will be interesting to study the $G_2$-Laplacian flow with more general symmetries, hoping to obtain long time existence and convergence.\\
	
	The outline of the article is as the following. In section \ref{hypersympletic-flow-of-simple-type}, we introduce \emph{hypersymplectic structures of simple type} on $\mathbb{T}^4$ and write out their evolution equation as a system of three scalar functions assuming all the structures during the flow remain of the simple type. However, since it is not a priori clear if the flow remains of simple type, we have to prove the short time existence of such flow instead of applying the general existence theorem in \cite{BX}. Unfortunately, the system is degenerated parabolic and there is no general theory about the existence of solutions. Instead of solving this, we transform it to an equivalent system \eqref{degenerate-system}. By integrating the third equation and substituting back in to the first two, we get a differential-integral parabolic system of two functions, for which we can adapt standard techniques in PDE to prove the short time existence of the solution in section \ref{section-short-time-existence}.  In section \ref{section-quasi-isometric}, by using the maximum principle we show the solution is bounded in $C^0$, which geometrically means that all the metrics along the flow are quasi-isometric to the standard flat one. In section \ref{section-scalar-curvature-increasing}, we derive an important evolution inequality for the scalar curvature, which gives us the decaying behavior of the scalar curvature. Then, by a \emph{blow up argument}, we show long time existence (the argument here is independent of the general one in \cite{FY} and is much simpler) in section \ref{section-long-time-existence}. In section \ref{section-convergence}, we use all the bounds obtained in the previous sections to show that the pulling back of the hypersymplectic flow by a family of diffeomorphisms (determined by the flow) converges to the standard hyperK\"ahler structure.

	\section{Hypersymplectic flow of simple type on $4$-torus}\label{hypersympletic-flow-of-simple-type}
	%\subsection{}
	
	Let $x_0, x_1, x_2, x_3$ be the standard coordinates on the standard $\mathbb{T}^4= S^1\times \mathbb{T}^3$, let $\underline{\omega}^0=(\omega_1^0, \omega_2^0, \omega_3^0)$ where
	
	\begin{equation}
		\begin{split}
			\omega_1^0= \dd x_0\wedge\dd x_1+ \dd x_2\wedge\dd x_3\\
			\omega_2^0= \dd x_0\wedge\dd x_2+ \dd x_3\wedge\dd x_1\\
			\omega_3^0= \dd x_0\wedge\dd x_3+ \dd x_1\wedge\dd x_2
		\end{split}
	\end{equation}\\
	are the triple of symplectic forms inducing the standard hyperK\"ahler structure on it. The corresponding 
	complex structures are written as $I_1, I_2, I_3$. 
	
	Let $\mathbb{T}^3$ act on the factor $\mathbb{T}^3\subset \mathbb{T}^4$ canonically, then every $\mathbb{T}^3$-invariant 2-form on $\mathbb{T}^4$ can be written as 
	\[
	\Omega=\sum_{i=1}^3 A_i\dd x_0\wedge\dd x_i 
	+ B_1\dd x_2\wedge\dd x_3 + B_2 \dd x_3\wedge\dd x_1
	+ B_3 \dd x_1\wedge \dd x_2
	\]
	for some smooth functions $A_1, A_2, A_3, B_1, B_2, B_3$ only depending on the $S^1$ factor of $\mathbb{T}^4$. If we assume $\dd\Omega=0$, then the 
	components $B_1,B_2$ and $B_3$ must be constant functions. 
	
	One simple case of these hypersymplectic structures arises when we assume 
	$$\omega_i =\omega^0_i - \dd I_i \dd \phi_i $$
	for three real valued functions $\phi_1, \phi_2, \phi_3$. If further we assume each $\phi_i$ only depends on the variable $x_0$,
	$\underline\omega$ is called a \emph{hypersymplectic structures of simple type}.  This type has a  nice formula:
	
	\begin{equation}\label{specialtriple}
		\begin{split}
			\omega_1= (1+ \phi_1'')\dd x_0\wedge\dd x_1+ \dd x_2\wedge\dd x_3\\
			\omega_2= (1+\phi_2'')\dd x_0\wedge\dd x_2+ \dd x_3\wedge\dd x_1\\
			\omega_3= (1+\phi_3'')\dd x_0\wedge\dd x_3+ \dd x_1\wedge\dd x_2
		\end{split}
	\end{equation}\\
	The condition of \emph{definiteness} for the hypersymplectic structure is that $1+\phi_i''>0, \; i=1, 2, 3.$
	Write for $i=1, 2, 3$, 
	\begin{equation} \label{EqnAVf}
		A_i= 1+ \phi_i'', \;\; V= (A_1A_2A_3)^\frac{1}{3}, \;\;f_i=\frac{A_i}{V} \end{equation}
	then the corresponding metric $g$ of this hypersymplectic structure (as defined in Equation \eqref{metric-from-definite-triple}) is given as:
	%\begin{equation}\label{metric}
	%(g_{ij}) =\frac{1}{(A_1A_2A_3)^\frac{1}{3}}\left(\begin{array}{cccc}
	%A_1A_2A_3 & 0 & 0 & 0\\
	%0      & A_1 & 0 & 0\\
	%0    & 0 & A_2 & 0\\
	%0    & 0 & 0 & A_3
	%\end{array}\right)
	%=
	%\left(\begin{array}{cccc}
	%V^2 & 0 & 0 & 0\\
	%0      & f_1 & 0 & 0\\
	%0    & 0 & f_2 & 0\\
	%0    & 0 & 0 & f_3
	%\end{array}\right)
	%\end{equation}
	
	\begin{equation}
		g = V^2 \dd x_0^2 + f_1 \dd x_1^2 + f_2 \dd x_2^2 + f_3\dd x_3^2
	\end{equation} 
	and the volume form is 
	\begin{equation} 
		\mu = V \dd x_{0123}
	\end{equation}
	It is a \emph{multiply warped product}.  The matrix of inner-product $Q=(Q_{ij})$ is given as
	\[
	(Q_{ij}) 
	=\frac{1}{(A_1A_2A_3)^\frac{1}{3}}\left(\begin{array}{ccc}
	A_1 & 0 & 0\\
	0 & A_2 & 0\\
	0 & 0 & A_3
	\end{array}\right)
	=
	\left(\begin{array}{ccc}
	f_1 & 0 & 0\\
	0 & f_2 & 0\\
	0 & 0 & f_3
	\end{array}\right)
	\]
	
	The crucial identity 
	\begin{equation}\label{f_1 f_2 f_3=1}
		f_1 f_2 f_3 \equiv 1
	\end{equation}
	will be used frequently.

	By using the definition of the Hodge star operator $*$, i.e. $\alpha\wedge *\alpha=|\alpha|^2 \mu$ for any form $\alpha$, we derive that:
	\[
	*\dd x_0=V^{-1}\dd x_{123}\; ,\;
	*\dd x_1 =-f_1^{-1}V\dd x_{023}\; ,\;
	*\dd x_2 =- f_2^{-1} V\dd x_{031}\; ,\;
	*\dd x_3 = - f_3^{-1} V\dd x_{012}.
	\]
	
	We then get the three torsion $1$-forms from Equation (\ref{torsion-one-form}):
	\[
	\tau_i = (\log \frac{A_i}{V})' \frac{A_i}{V^2}\dd x_i
	\]
	for $ i=1, 2, 3$, where `` $'$ '' denotes the derivative with respect to $x_0$. 
	Let
	\[
	\omega_i(0) =\omega^0_i - \dd I_i \dd \phi_i(0)
	\]
	be a smooth \emph{hypersymplectic structure of simple type} on $\mathbb{T}^4$.  The \emph{ hypersymplectic flow}  (Equation (\ref{definite-triple-flow})) 
	with initial data $\underline\omega(0)=(\omega_1(0), \omega_2(0), \omega_3(0))$ , 
	
	\begin{equation}\label{definite-triple-flow-of-simple-type}
		\begin{split}
			\left\{\begin{array}{ll}
				\partial_t \omega_i = \dd \tau_i, & i=1, 2, 3\\
				\underline\omega|_{t=0}=\underline\omega(0).
			\end{array}\right.
		\end{split}
	\end{equation}
	is then reduced to a system of PDE:
	
	\begin{equation}\label{EqnA}
		\partial_t A_i =\Big((\log \frac{A_i}{V})'\frac{A_i}{V^2}\Big)', \; \; i=1, 2, 3
	\end{equation}
	for the three unknown functions $A_i: S^1\times\mathbb{R}^+\rightarrow \mathbb{R}^+, \;\; i=1, 2, 3$, whose initial data satisfies the normalization condition $\int_{S^1} A_i(\theta, 0)\dd \theta=2\pi$,  where $S^1=\mathbb{R}/2\pi \mathbb{Z}$.

	\subsection{Evolution equations}
	
	The PDEs \eqref{EqnA} for $A_i$'s can be expanded as
	
	\begin{equation}
		\begin{split}
			\partial_t A_i 
			& = \frac{VA_i'' - V'' A_i}{V^3} - \frac{3 A_i' V'}{V^3} + \frac{3 A_i V'^2}{V^4}\\
			& = \frac{1}{3(A_1A_2A_3)^\frac{2}{3}}\Big( 2 A_i'' - \sum_{j\neq i} \frac{A_i}{A_j} A_j''\Big) \\
			&\qquad +  \frac{1}{3(A_1A_2A_3)^\frac{2}{3}}\Big( A_i \sum_{j=1}^3 \frac{A_j'^2}{A_j^2} 
			+ \big(\sum_{j=1}^3 \frac{A_j'}{A_j}\big)^2 - 2 A_i' \sum_{j=1}^3 \frac{A_j'}{A_j}\Big)
		\end{split}
	\end{equation}
	for $ i=1, 2, 3$. 
	The principal symbol of the differential operator of the above PDE system is 
	
	\begin{equation}
		\begin{split}
			\frac{1}{3(A_1A_2A_3)^\frac{2}{3}}\left(\begin{array}{ccc}
				2   &   -\frac{A_1}{A_2}   &  -\frac{A_1}{A_3}\\
				-\frac{A_2}{A_1}  &  2   &  - \frac{A_2}{A_3}\\
				-\frac{A_3}{A_1} & -\frac{A_3}{A_2} & 2
			\end{array}\right)
		\end{split}
	\end{equation}
	Therefore, the system is a quasi-linear parabolic equation of degenerate type, which simply means the principal symbol is not invertible. This could easily be seen by the fact that the evolution for $\sum_{i=1}^3 \frac{1}{A_i}\partial_t A_i = \partial_t \log V$ does not involve second derivatives in the spatial direction (see Equation \eqref{EqnV}).

	%\begin{question}
	%The short time existence of Equation (\ref{EqnA}) follows from \cite{FY} which uses the short time existence of \cite{BX} and symmetry 
	%{\color{red} \Big\{only $\mathbb{T}^3$ symmetry is preserved, why the forms are preserved to be of simple type?
	%On idea is to use symmetry. Define a $3$-form 
	%$$\Psi= \omega_i\wedge\dd x_i - 3 \dd x_{123}$$
	%Then the symmetry of the coefficient parameter matrices $(\alpha_1,\alpha_2, \alpha_3)$ where $\omega_i=\dd x_0\wedge\alpha_i +\beta_i$ is reflected as $\Psi\equiv 0$. \Big\}}. 

	An important quantity\footnote{This equals to twice the quantity $|\mathbf{T}|^2$ in \cite{FY}.} is 
	\begin{equation}\label{Tformula}
		\mathcal{T}=\tr \Big(Q^{-1}\langle \tau, \tau\rangle\Big)=V^{-2}\sum_{i=1}^3 \Big((\log f_i)'\Big)^2
	\end{equation}
	
	The evolution equation for $V$ is rather simple: 
	
	\begin{lemma} 
		\begin{equation}\label{EqnV}
			\partial_t V = \frac{1}{3}\mathcal{T}V
			%=\frac{1}{3V}\sum_{i=1}^3 \Big((\log \frac{A_i}{V})'\Big)^2
		\end{equation}
	\end{lemma}
	\begin{proof}
		Notice the relationship between the hypersymplectic flow on $\mathbb{T}^4$ and the $G_2$-Laplacian flow on $\mathbb{T}^7$ \cite[Lemma 2.9]{FY}, and the relationship between the volume forms of the corresponding metrics on $4$-dimension and $7$-dimension \cite[Lemma 2.5]{FY}. The general evolution equation for the volume form in $G_2$-Laplacian flow \cite[Equation 3.8]{LW1} gives the stated result.   
	\end{proof}
	
	Calculating from equations \eqref{EqnA} and \eqref{EqnV} we get the evolution equations for $f_1, f_2, f_3$. 
	
	\begin{lemma}\label{f_ievolution-lemma}
		%   \label{f_ievolution-lemma}
		\[
		\partial_t f_i
		= \frac{1}{V}(\frac{f_i''}{V}- \frac{f_i'V'}{V^2}) -\frac{1}{3}f_i \mathcal{T}
		\;\;, \;\; i=1, 2, 3
		\]     
	\end{lemma}   
	\begin{proof}
		Since $A_i = f_i V$, we have 
		\[
		A_i' = f_i' V + f_i V'\;\;\;, 
		\;\;\;
		A_i''= f_i'' + 2 f_i' V' + f_i V''
		\]
		and therefore 
		\begin{equation}
			\begin{split}
				\partial_t f_i 
				& = \frac{\partial_t A_i}{V}- \frac{A_i}{V^2}\partial_t V \\
				& = \frac{1}{V}\Big(\frac{A_i''}{V^2} - 3 \frac{V'}{V^3}A_i' - (\frac{V'}{V^3})' A_i
				\Big) - \frac{1}{3}f_i \mathcal{T}\\
				& = \frac{1}{V}\Big(\frac{f_i'' + 2 f_i' V' + f_i V''}{V^2} - 3 \frac{V'}{V^3}(f_i' V + f_i V') - (\frac{V''}{V^3} - \frac{3V'^2}{V^4}) f_i V
				\Big) - \frac{1}{3}f_i \mathcal{T}\\
				& = \frac{1}{V}(\frac{f_i''}{V}- \frac{f_i'V'}{V^2}) -\frac{1}{3}f_i \mathcal{T}
			\end{split}
		\end{equation}
	\end{proof}      
	
	Because of the identity $f_1 f_2 f_3\equiv 1$, the three equations in Lemma \ref{f_ievolution-lemma} are not totally independent of each other. We replace the equation for $f_3$ by the equation for $V$ to get the following system:

	\begin{equation}\label{degenerate-system}
		\left\{\begin{array}{l}
			\partial_t f_1
			= \frac{1}{V}\Big(\frac{f_1''}{V}- \frac{f_1'V'}{V^2}\Big) -\frac{1}{3}f_1 \mathcal{T}\\
			\\
			\partial_t f_2
			= \frac{1}{V}\Big(\frac{f_2''}{V}- \frac{f_2'V'}{V^2}\Big) -\frac{1}{3}f_2 \mathcal{T}\\
			\\
			\partial_t V
			= \frac{1}{3}\mathcal{T}V
		\end{array}\right.
	\end{equation}
	where $\mathcal{T}=V^{-2}\Big\{\sum_{i=1}^2 \Big((\log f_i)'\Big)^2 + \Big((\log f_1 +\log f_2)'\Big)^2\Big\}$. 
	
	\begin{lemma}\label{equivalence-between-systems}
		Let $f_1, f_2, V\in C^2(S^1\times [0, t_0); \mathbb{R}^+)$ be three functions satisfying the PDE system \eqref{degenerate-system}, then the functions $A_1= f_1 V, A_2=f_2 V, A_3 = f_1^{-1} f_2^{-1} V$ satisfy the PDE system \eqref{EqnA}. 
	\end{lemma}
	\begin{proof}
		From the data $f_1, f_2, V$ satisfying the equations \eqref{degenerate-system}, denote $f_3=f_1^{-1}f_2^{-1}$, then direct calculation shows that 
		\begin{equation}
			\begin{split}
				\frac{1}{V} &\Big(\frac{f_3''}{V}- \frac{f_3'V'}{V^2}\Big) -\frac{1}{3}f_3 \mathcal{T}\\
				& = V^{-2}(-f_1^{-2} f_2^{-1} f_1'' - f_1^{-1} f_2^{-2} f_2'' + 2 f_1^{-3}f_2^{-1} f_1'^2 + 2f_1^{-1} f_2^{-3} f_2'^2 + 2f_1^{-2}f_2^{-2} f_1' f_2') \\
				&  \quad + V^{-3}V'(f_1^{-2} f_2^{-1} f_1' + f_1^{-1} f_2^{-2} f_2')
				-\frac{1}{3}f_3 \mathcal{T}\\
				& = -f_1^{-2} f_2^{-1} V^{-1}(V^{-1}f_1'' - V^{-2}V' f_1')
				-f_1^{-1} f_2^{-2} V^{-1}(V^{-1}f_2'' - V^{-2}V' f_2') 
				+ \frac{2}{3} f_1^{-1} f_2^{-1}\mathcal{T}\\
				& = \partial_t f_3       
			\end{split}
		\end{equation}
		Then for $i=1, 2, 3$, we verify that 
		\begin{equation}
			\begin{split}
				\partial_t A_i
				& =V \partial_t f_i  + f_i \partial_t V
				= \frac{f_i''}{V} - \frac{f_i'V'}{V^2}\\
				& = \frac{1}{V} (\frac{A_i'}{V}- \frac{A_iV'}{V^2})' 
				- \frac{V'}{V^2}(\frac{A_i'}{V}- \frac{A_iV'}{V^2})\\
				& = \Big(\frac{A_i''}{V^2} - 3 \frac{V'}{V^3}A_i' - (\frac{V'}{V^3})' A_i
				\Big) \\
				& = \Big((\log \frac{A_i}{V})'\frac{A_i}{V^2}\Big)'
			\end{split}
		\end{equation}
	\end{proof}

	Though \eqref{degenerate-system}  is not parabolic, the next proposition shows that the system \eqref{degenerate-system} has short time solution for any smooth positive initial functions $f_1, f_2$ and $V$ on $S^1$.

	\subsection{Short time existence}\label{section-short-time-existence}
	%From now on to the end of Lemma \ref{lem general existence of the integral equation},  in the notation of the function space, we hide the $S^{1}$ when the domain of the underlying function is $S^{1}$ or $S^{1}\times I$ ($I$ is a time interval).
	We use $C^{1,\alpha}(S^1), \; C^{2,\alpha}(S^1)$ etc. to denote the Banach space with the usual Schauder norm, and $C^{\alpha, \frac{\alpha}{2}}(S^1\times [0, t_0]), \; C^{1+\alpha, \frac{1+\alpha}{2}}(S^1\times [0, t_0]), \; C^{2+\alpha, 1+\frac{\alpha}{2}}(S^1\times [0, t_0])$ etc. to denote the Banach space with the parabolic Schauder norm defined in  \cite[Section 1 in Chapter 1]{LSU}. The following proposition and Lemma \ref{equivalence-between-systems} indicates that the hypersymplectic flow \eqref{definite-triple-flow-of-simple-type} exists for a short time and preserves the simple type condition. 
	\begin{proposition}[(Short time existence)]\label{prop short time existence} 
		For any $0<\alpha^{\prime}<\alpha<1$, and  positive initial functions $f_{i}(\cdot, 0)\in C^{2,\alpha}(S^{1})$, $i=1,2$, $V(\cdot, 0)\in C^{1,\alpha}(S^{1})$,   there exists an $\epsilon>0$ and $c_{0}$ depending on
		\begin{itemize}
			\item  $|f_{i}(\cdot, 0)|_{C^{2,\alpha}(S^{1})}, \;\; i=1,2$, \;\; $|V(\cdot, 0)|_{C^{1,\alpha}(S^{1})}$;
			\item  the  positive lower bounds of the above functions;
			\item   $\alpha^{\prime},\ \alpha$, 
		\end{itemize} such that  \eqref{degenerate-system} initiated from the functions admits a unique solution in $C^{2+\alpha^{\prime},1+\frac{\alpha^{\prime}}{2}}(S^{1}\times [0,\epsilon])$ with norm $\leq c_{0}$. 
	\end{proposition}
	
	We divide the first two equations in \eqref{degenerate-system} by $f_{i}$ respectively to find
	\begin{equation}\label{evolution-log-f_i}
		\frac{\partial \log{f_i}}{\partial t}= \frac{1}{V^{2}}\Big\{(\log f_i)^{\prime\prime}+\Big( (\log f_i)^{\prime}\Big)^{2}-\frac{1}{2}(\log f_i)^{\prime}(\log V^{2})^{\prime}\Big\} -\frac{1}{3}\mathcal{T}
	\end{equation}
	We will formulate Proposition \ref{prop short time existence}  more generally as Proposition \ref{prop general abstract short time existence}. 
	
	% the proof of Proposition \ref{prop short time existence} can be sketched by the following diagram:
	%\begin{center}
	%	\begin{tikzpicture}[->,>=stealth',shorten >=1pt,auto,node distance=2.5cm,
	%	thick,main node/.style={rectangle,draw}]
	%	\node[main node] (1)  {Prop. \ref{prop short time existence}} ;
	%	\node[main node] (2) [left of=1] {Prop. \ref{prop general abstract short time existence}} ;
	%	\node[main node] (3) [left of=2] {Lem. \ref{lem chat is weaker than c}} ;
	%	\node[main node] (4) [above left of=2] {Lem. \ref{lem general existence of the integral equation}} ;  \node[draw,align=left] at (0,1.5) {The arrows\\ mean implying.};
	%	\path[every node/.style={font=\sffamily\small}]
	%	(2) edge node [right] {} (1)
	%	(3) edge node [right] {} (2)
	%	(4) edge node [right] {} (2);
	%	\end{tikzpicture}
	%\end{center}
	Besides standard arguments, there are two crucial observations for Proposition \ref{prop short time existence} (and \ref{prop general abstract short time existence}). 
	\begin{itemize}\item The third equation in \eqref{degenerate-system} can be integrated with respect to $t$,  i.e. it yields
		\begin{equation}\label{equ V2 integral formula}
			%	\frac{\partial V}{\partial t}=\frac{K}{3V}\Rightarrow
			V^{2}(x_0, t)= \frac{2}{3}\int^{t}_{0}K\dd s+V^{2}(x_{0},0)
		\end{equation}
		where $K=\mathcal{T}V^{2}=\Big\{\sum_{i=1}^2 \Big((\log f_i)'\Big)^2 + \Big((\log f_1 +\log f_2)'\Big)^2\Big\}.$
		\item  In the linearization of the flow equation in the form of \eqref{equ the integral equation}, the integral terms are ``small" in some sense (see Lemma \ref{lem general existence of the integral equation}).
	\end{itemize}

	Let $w_{i}(x_0, t)=\log f_i (x_0, 0)\;\; (\forall t)$, then $u_{i}\triangleq \log{f_i}-w_{i}=0$ when $t=0$. Moreover, we can understand \eqref{equ V2 integral formula} as $V^{2}=V^{2}(u)$.  Let \begin{eqnarray}
		K_{0}(u^{\prime}) 
		& \triangleq & 2(u^{\prime}_{1})^{2}+2(u^{\prime}_{2})^{2}+4u^{\prime}_{1}w^{\prime}_{1}+4u^{\prime}_{2}w^{\prime}_{2}+2u^{\prime}_{1}w^{\prime}_{2}+2u^{\prime}_{2}w^{\prime}_{1}+2u^{\prime}_{1}u^{\prime}_{2}\\
		\widehat{w} 
		& \triangleq &  
		2\big(w^{\prime}_{1}w^{\prime}_{2}+(w^{\prime}_{1})^{2}+(w^{\prime}_{2})^{2}\big),
		\label{equ w hat}\\
		F
		& =& 
		\frac{2}{3}\int^{t}_{0}\widehat{w}\dd s+V^{2}(x_{0},0)=\frac{2}{3}t\widehat{w}+V^{2}(x_{0},0)\ \\
		& & \qquad \textrm{(which is the value of}\ V^{2}\ \textrm{when}\ u=0), \nonumber\\
		H_{i}
		& \triangleq &
		w^{\prime\prime}_{i}+(w^{\prime}_{i})^{2}-\frac{\widehat{w}}{3}-\frac{(\log F)^{\prime}}{2}, \; \; i=1, 2.\label{equ w underline}
	\end{eqnarray}
	The first two equations in \eqref{degenerate-system} are equivalent to 
	\begin{eqnarray}
		\mathscr{P}(u_{i}) \nonumber
		& \triangleq& V^{2}(u)\frac{\partial u_{i}}{\partial t}
		-\left\{u^{\prime\prime}_{i}+(u^{\prime}_{i})^{2}+2u^{\prime}_{i}w^{\prime}_{i}
		-\frac{1}{2}u_i^{\prime} \left( \log V^{2}(u) \right)^{\prime}
		-\frac{1}{2}w_i^{\prime} \left( \log \frac{V^{2}(u)}{F} \right)^{\prime}\right\}\\
		& & \qquad +\frac{K_{0}(u^{\prime})}{3} \nonumber\\
		& = &H_{i}, \; i=1, 2. \label{equ the integral equation}
	\end{eqnarray}
	Moreover, $K=K_{0}(u^{\prime})+\widehat{w}$. Hence \eqref{equ the integral equation} becomes an autonomous differential-integral equation, it suffices to solve it in $C_{0}^{2+\alpha,1+\frac{\alpha}{2}}$ (the subspace of $C^{2+\alpha,1+\frac{\alpha}{2}}$ consisting of those whose initial values are  $0$). 
	
	\begin{definition} We  define the weaker norm which only depends on the $C^{\alpha,\frac{\alpha}{2}}-$norms of the spatial derivatives of $u$:
		\[
		|u|_{\widehat{C}^{2+\alpha,\frac{1+\alpha}{2}}(S^{1}\times [0,t_{0}])}\triangleq |u|_{C^{\alpha,\frac{\alpha}{2}}(S^{1}\times [0,t_{0}])}+|u^{\prime}|_{C^{\alpha,\frac{\alpha}{2}}(S^{1}\times [0,t_{0}])}+|u^{\prime\prime}|_{C^{\alpha,\frac{\alpha}{2}}(S^{1}\times [0,t_{0}])}\]
		%We note that the $\widehat{C}-$norm is weaker than the $C-$norm. 
	\end{definition}

	\begin{definition}\label{Def admissible nonlinear operators} We say that $\mathscr{Q}$ is an $\alpha-$admissible differential-integral operator on vector-valued functions $u\in C^{\infty} ( S^{1}\times [0,t_{0}], \R^{n})$ if 
		\begin{equation} \label{equ Def admissible nonlinear operators}
			\begin{split} 
				\mathscr{Q}(u)=    
				& \left( k_{1}+\int^{t}_{0}K_{1}(u^{\prime},u)\dd s\right)\frac{\partial u}{\partial t}-u^{\prime\prime}+b_{1}(u^{\prime},u)+b_{2}(u^{\prime},u)\left\{ \log\left( k_{0}+\int^{t}_{0}K_{0}(u^{\prime},u)\dd s\right) \right\}^{\prime}\nonumber
				\\ & +b_{0}(u^{\prime},u)\left\{\log\left( 1+\frac{\int^{t}_{0}K_2(u^{\prime},u)\dd s}{k_{2}}\right)\right\}^{\prime}\nonumber
			\end{split}
		\end{equation}
		for some 
		\begin{itemize}
			\item $\mathbb{R}^n$-valued polynomials $b_{0}(x,y), b_1(x,y), b_2(x, y)$, and scalar valued polynomials $K_0(x,y)$, $K_1(x, y)$, $K_2(x, y)$, whose coefficients are functions in $C^{\alpha,\frac{\alpha}{2}}(S^{1}\times [0,t_{0}])$, such that  $b_{1}(0,0)=b_2(0,0)=0$ and $K_2(0,0)=0 $;
			\item positive functions $k_{0},k_{1},k_{2}\in C^{\alpha,\frac{\alpha}{2}}(S^{1}\times [0,t_{0}])$.
		\end{itemize}
	\end{definition}
	\begin{remark}The difference between the last two terms in the formula of the above definition is that  the polynomial with no zeroth-order term  in $b_{2}(u^{\prime},u)\Big\{\log\big(k_{0}+\int^{t}_{0}K_{0}(u^{\prime},u)\dd s\big)\Big\}^{\prime}$ is outside the integral, while the one in the other term is inside. 
	\end{remark}
	Let  $L_{u}$ be the linearization of $\mathscr{Q}$ at $u$. Mainly because the co-efficient of $\frac{\partial u}{\partial t}$ in $\mathscr{Q}(u)$ only depend on the spatial derivatives of $u$, it's routine to verify for any $\alpha,\ t_{1} \in (0,1)$ that 
	\begin{equation}\label{equ solution to the linear equation is an approximate solution to the nonlinear one}
		|\mathscr{Q}(u)-L_{0}u|_{C^{\alpha,\frac{\alpha}{2}}(S^{1}\times [0,t_{1}])}\leq C|u|_{\widehat{C}^{2+\alpha, 1+ \frac{\alpha}{2}}(S^{1}\times [0,t_{1}])}|u|_{C^{2+\alpha, 1+\frac{\alpha}{2}}(S^{1}\times [0,t_{1}])}
	\end{equation}
	\begin{remark} From now on, none of the constants ``$C$'' depends on $\epsilon$ or $t_{0}$. 
	\end{remark}

	Because $\frac{\partial u}{\partial t}$ is not included in the $\widehat{C}-$norm, it's a  routine exercise to obtain
	\begin{lemma}\label{lem chat is weaker than c}When $t_{1}\leq \epsilon$ and $0<\alpha^{\prime}<\alpha<1$, we have for any $h$ that 
		\[
		|h|_{\widehat{C}_{0}^{2+\alpha^{\prime}, 1+ \frac{\alpha^{\prime}}{2}}(S^{1}\times [0,t_{1}])}\leq C\epsilon^{\frac{\alpha-\alpha^{\prime}}{2}}|h|_{C_{0}^{2+\alpha, 1+ \frac{\alpha}{2}}(S^{1}\times [0,t_{1}])}
		\]
	\end{lemma}
	\begin{proof}Under the  conditions, it suffices to apply the following  elementary inequality (which holds for any $v\in C_{0}^{\alpha,\frac{\alpha}{2}}$) to $h,\ h^{\prime},\ h^{\prime\prime}$.
		\begin{equation}\label{equ lem chat is weaker than c}
			|v|_{C_{0}^{\alpha^{\prime},\frac{\alpha^{\prime}}{2}}(S^{1}\times [0,t_{1}])}\leq C\epsilon^{\frac{\alpha-\alpha^{\prime}}{2}}|v|_{C_{0}^{\alpha,\frac{\alpha}{2}}(S^{1}\times [0,t_{1}])}
		\end{equation}
	\end{proof}
	\begin{remark} The above lemma is the only reason why we have to decrease the  ``$\alpha$'' a little bit. 
	\end{remark}
	\begin{lemma}\label{lem general existence of the integral equation}  For any $\alpha\in (0,1)$ and $t_{0}>0$, suppose $L$ is a linear differential-integral operator of a $\R^{n}$(vector)-valued function $h$ with the following formula,
		\[
		L(h)=U_{0}\frac{\partial h}{\partial t}-h^{\prime\prime}+A_{0}(h^{\prime},h)+\sum_{k=1}^{m}B_{k}\cdot \int_{0}^{t}A_{k}(h^{\prime\prime},h^{\prime},h)\dd s,\]
		where 
		\begin{itemize}
			\item  $U_{0}$ is an arbitrary scalar positive function in $C^{\alpha,\frac{\alpha}{2}}(S^{1}\times[0,t_{0}])$,
			\item the $B_{k}$'s $(1\leq k\leq m)$ are arbitrary $\R^{n}-$valued functions in  $C^{\alpha,\frac{\alpha}{2}}(S^{1}\times[0,t_{0}])$,
			\item the $A_{k}$'s $(0\leq k\leq m)$ are arbitrary $\R^{n}-$valued linear polynomials (in the entries of the vector variables) with $C^{\alpha,\frac{\alpha}{2}}(S^{1}\times[0,t_{0}])$ coefficients,
			\item the ``$\cdot$'' (in the last term) means the usual inner product of $\R^{n}$. 
		\end{itemize} 
		Then there exist
		\begin{itemize}
			\item  an $\epsilon$  depending on $t_{0}$, the positive lower bound of $U_{0}$ over $S^{1}\times [0,t_{0}]$, the $C^{\alpha,\frac{\alpha}{2}}(S^{1}\times[0,t_{0}])-$norms of $U_{0}$, the $B_{k}$'s,  the coefficients of the polynomials $A_{k}$ ($0\leq k\leq m$),
			\item  and a $C$ depending on all the above data except $t_{0}$,
		\end{itemize}	 such that for any $t_{1}\leq  \epsilon$,
		\begin{enumerate}
			\item $L$ admits a bounded inverse $L^{-1}$ from $C^{\alpha,\frac{\alpha}{2}}(S^{1}\times[0,t_{1}])$ to $C_{0}^{2+\alpha,1+\frac{\alpha}{2}}(S^{1}\times[0,t_{1}])$;
			\item the norm of $L^{-1}$ is $\leq C$. 
		\end{enumerate}
	\end{lemma}
	\begin{proof}Because the terms in the sum $\sum_{k=1}^{m}B_{k}\cdot \int_{0}^{t}A_{k}(h^{\prime\prime},h^{\prime},h)\dd s$ are similar to each other for our purpose, and the highest-order derivative involved is
		$h^{\prime\prime}$,  without loss of generality, we assume $m=1$ and $A_{1}(h^{\prime\prime},h^{\prime},h)=h^{\prime\prime}$. The crucial observation is that the integral term is small when $\epsilon$ is small i.e.
		\begin{eqnarray}
			& &|\int_{0}^{t}h^{\prime\prime}\dd s|_{C^{\alpha,\frac{\alpha}{2}}(S^{1}\times [0,t_{1}])}\leq 2\epsilon^{1-\frac{\alpha}{2}}|h|_{C_{0}^{2+\alpha,1+\frac{\alpha}{2}}(S^{1}\times [0,t_{1}])}\label{equ 1 in lem general existence of the integral equation}
		\end{eqnarray}
		
		Assuming the above  inequality, the proof is complete by the invertibility of the linear parabolic differential operator $U_{0}\frac{\partial h}{\partial t}-h^{\prime\prime}+A_{0}(h^{\prime},h)$, and Theorem 17.6 in \cite{GT}.

		To prove  \eqref{equ 1 in lem general existence of the integral equation}, let $d_{S^{1}}$ be the intrinsic distance of $S^{1}$, it suffices to observe that for any $t,\ t_{3},\ t_{2}\leq t_{1}$,  
		\begin{equation}
			\begin{split}
				\frac{|\int_{0}^{t_{3}}h^{\prime\prime}(x_{0},s)\dd s-\int_{0}^{t_{2}}h^{\prime\prime}(x_{0},s)\dd s|}{|t_{3}-t_{2}|^{\frac{\alpha}{2}}}
				& \leq |t_{3}-t_{2}|^{1-\frac{\alpha}{2}}|h^{\prime\prime}|_{C^{0}(S^{1}\times [0,t_{1}])}\\
				& \leq \epsilon^{1-\frac{\alpha}{2}}|h|_{C_{0}^{2+\alpha,1+\frac{\alpha}{2}}(S^{1}\times [0,t_{1}])}
			\end{split}
		\end{equation}
		\begin{equation}
			\frac{|\int_{0}^{t}h^{\prime\prime}(x_{0,1},s)\dd s-\int_{0}^{t}h^{\prime\prime}(x_{0,2},s)\dd s|}{d^{\alpha}_{S^{1}}(x_{0,1},x_{0,2})}\leq t|h|_{C_{0}^{2+\alpha,1+\frac{\alpha}{2}}(S^{1}\times [0,t_{1}])}\leq \epsilon |h|_{C_{0}^{2+\alpha,1+\frac{\alpha}{2}}(S^{1}\times [0,t_{1}])}
		\end{equation}
		
		Though  standard, for the reader's convenience, we still say something about the invertibility of the linear parabolic differential operator $U_{0}\frac{\partial h}{\partial t}-h^{\prime\prime}+A_{0}(h^{\prime},h)$. 
		Using the path 
		$$L_{\lambda}\triangleq \left( (1-\lambda)U_{0}+\lambda\right) \frac{\partial h}{\partial t}-h^{\prime\prime}+(1-\lambda)A_{0}(h^{\prime},h),\ \lambda\in [0,1] $$
		connecting it to the standard heat operator $\frac{\partial h}{\partial t}-h^{\prime\prime}$, its invertibility  follows from the a priori estimate in \cite[Theorem 3]{Sch}, the invertibility of the heat operator, and the continuity method in \cite[Theorem 5.2]{GT}.
	\end{proof}
	\begin{proposition}[(General version of Proposition \ref{prop short time existence})] \label{prop general abstract short time existence}For any $0<\alpha^{\prime}<\alpha<1$,   any $\alpha-$admissible operator $\mathscr{Q}$, and any $f\in C^{\alpha,\frac{\alpha}{2}}(S^{1}\times [0,t_{0}])$ ($t_{0}>0$), there exists 
		\begin{itemize}
			\item an  $\epsilon>0$  depending on $t_{0}$, $|f|_{C^{\alpha,\frac{\alpha}{2}}(S^{1}\times [0,t_{0}])}$, $\mathscr{Q}$, $\alpha^{\prime},\ \alpha$,
			\item  and a $c_{0}$ depending on all the above data except $t_{0}$,\end{itemize}  such that the system 
		\begin{equation}\label{equ general abstract differential-integral system}\mathscr{Q}(u)=f,\ u(x_{0},0)=0 \end{equation}  
		admits a solution in $C^{2+\alpha^{\prime},1+\frac{\alpha^{\prime}}{2}}(S^{1}\times[0,\epsilon])$  with norm $\leq c_{0}$. Moreover, this solution is unique in $C^{2+\alpha^{\prime},1+\frac{\alpha^{\prime}}{2}}(S^{1}\times[0,t_{1}])$  for any $t_{1}\leq \epsilon$.
	\end{proposition}
	\begin{remark}\label{rmk independent of t0} In Lemma \ref{lem general existence of the integral equation} and Proposition \ref{prop general abstract short time existence}, except that $\epsilon\leq t_{0}$, none of the bounds (constants)  depends on $t_{0}$. \end{remark}
	\begin{proof}[of Proposition \ref{prop general abstract short time existence}:] For any $u$, routine computations and $\alpha-$admissibility show that $L_{u}$ (the linearization of  $\mathscr{Q}$ at $u$) satisfies the conditions in Lemma \ref{lem general existence of the integral equation}, which will be 
		frequently implemented in the following (in \eqref{equ 2 proof of prop general abstract short time existence} for example). 
		
		We first solve the linear equation $L_{0}u=f$ by Lemma \ref{lem general existence of the integral equation}. The crucial observation is that, by \eqref{equ solution to the linear equation is an approximate solution to the nonlinear one} and Lemma \ref{lem chat is weaker than c}, when $\epsilon$ is small, $u$ is almost a  solution to \eqref{equ general abstract differential-integral system}. Then it suffices to perturb $u$  to obtain a genuine solution. For the reader's convenience, we still include the following standard argument which is a quantitative version of \cite[Theorem 17.6]{GT}.

		Let $\sigma\triangleq \mathscr{Q}(u)-L_{0}u$ and thus $\mathscr{Q}(u)=\sigma+f$, we then need to solve the equation:
		\begin{equation}\label{equ 0 proof of prop general abstract short time existence} \widetilde{\mathscr{Q}}(v,\sigma)\triangleq \mathscr{Q}(u+v)-\mathscr{Q}(u)+\sigma=0\ (\textrm{which can be trivially solved by}\ v=\sigma=0)
		\end{equation}
		
		We define
		the ``quadratic part'' of $\mathscr{Q}(u+v)$ in $v$ as the following:
		$$E_{u}(v)\triangleq \mathscr{Q}(u+v)- \left( \mathscr{Q}(u)+L_{u}(v)\right)\ (\textrm{   just subtract}\ \mathscr{Q} \ \textrm{by the linearization at}\ v=0)$$
		For any $t_{1}\leq \epsilon$, it's routine to verify via Definition \ref{Def admissible nonlinear operators} that 
		\begin{equation}\label{equ 1 proof of prop general abstract short time existence}
			|E_{u}(v)|_{C^{\alpha^{\prime},\frac{\alpha^{\prime}}{2}}(S^{1}\times [0,t_{1}])}\leq C|v|^{2}_{C_{0}^{2+\alpha^{\prime},1+\frac{\alpha^{\prime}}{2}}(S^{1}\times [0,t_{1}])}
		\end{equation}
		Then solving \eqref{equ 0 proof of prop general abstract short time existence} is equivalent to finding a fixed point of the map 
		\begin{equation} 
			\begin{split} 
				\mathscr{T} : C_{0}^{2+\alpha^{\prime},1+\frac{\alpha^{\prime}}{2}}(S^{1}\times[0,t_{1}])
				& \longmapsto  C_{0}^{2+\alpha^{\prime},1+\frac{\alpha^{\prime}}{2}}(S^{1}\times[0,t_{1}])\\
				v
				& \longmapsto 
				v-L^{-1}_{u}\widetilde{\mathscr{Q}}(v,\sigma)=-L^{-1}_{u}\sigma-L^{-1}_{u}E_{u}(v)
			\end{split}
		\end{equation}

		%	$$T(v)=v-L^{-1}_{u}\widetilde{\mathscr{Q}}(v,\sigma)=-L^{-1}_{u}\sigma-L^{-1}_{u}E_{u}(v).$$ 
		Hence, combining \eqref{equ solution to the linear equation is an approximate solution to the nonlinear one}, Lemma \ref{lem chat is weaker than c}, and \ref{lem general existence of the integral equation}, by the proof of \cite[Theorem 17.6]{GT}, there exists a $\delta$ depending quantitatively on the data in Proposition \ref{prop general abstract short time existence} such that $\mathscr{T}$ is contracting when $|v|_{C_{0}^{2+\alpha^{\prime},1+\frac{\alpha^{\prime}}{2}}(S^{1}\times [0,t_{1}])}\leq \delta$. Let $L=L_{u}$, then Lemma \ref{lem general existence of the integral equation}(i) yields the inverse $L_{u}^{-1}$. Moreover, \eqref{equ solution to the linear equation is an approximate solution to the nonlinear one}, Lemma \ref{lem chat is weaker than c}, and the bound on $L^{-1}_{u}$ in Lemma \ref{lem general existence of the integral equation}(ii) imply that 
		\begin{equation}\label{equ 2 proof of prop general abstract short time existence}
			|L^{-1}_{u}\sigma|_{C_{0}^{2+\alpha^{\prime},1+\frac{\alpha^{\prime}}{2}}(S^{1}\times [0,t_{1}])}
			\leq C_{u,f}\epsilon^{\frac{\alpha-\alpha^{\prime}}{2}}|f|^{2}_{C_{0}^{2+\alpha, 1+ \frac{\alpha}{2}}(S^{1}\times [0,t_{1}])}
		\end{equation}
		for $t_1\leq \epsilon$ (Note that  $|f|_{C_{0}^{2+\alpha, 1+ \frac{\alpha}{2}}(S^{1}\times [0,t_{1}])}$ is not assumed to be small). Though the constant $C_{u,f}$ above might depend on $u$ and $f$, it does not depend on $\epsilon$.
		
		Thus when $\epsilon$ is  small enough with respect to $\delta$,   the first iteration $\mathscr{T}(0)\ (=-L^{-1}_{u}\sigma)$ stays in this $\delta-$neighborhood of $0$. Then  \cite[Theorem 17.6]{GT} (when  ``$B_{1}$''=\; $C_{0}^{2+\alpha^{\prime},1+\frac{\alpha^{\prime}}{2}}(S^{1}\times [0,t_{1}])$, ``$B_{2}$''=``$X$''=$C^{\alpha^{\prime},\frac{\alpha^{\prime}}{2}}(S^{1}\times [0,t_{1}])$ yields the unique solution to \eqref{equ 0 proof of prop general abstract short time existence} (for  $\sigma= \mathscr{Q}(u)-L_{0}u$).
	\end{proof}

	\begin{proof}[of Proposition \ref{prop short time existence}:] First we note that, by \eqref{equ w hat}--\eqref{equ w underline}, for $i=1,2$ and $\alpha\in (0,1)$, the right hand side of \eqref{equ the integral equation} satisfies 
		\begin{eqnarray}
			&	|H_{i} & |_{ C^{\alpha,\frac{\alpha}{2}}(S^{1}\times [0, t_0])}\nonumber
			\\& \leq &  |w_{1}^{\prime\prime}|_{C^{\alpha}(S^{1})}+ |w_{2}^{\prime\prime}|_{C^{\alpha}(S^{1})}+C |w_{1}^{\prime}|^{2}_{C^{\alpha}(S^{1})}+ C|w_{2}^{\prime}|^{2}_{C^{\alpha}(S^{1})}\nonumber 
			\\& & \quad +C|\big( \log \big(\frac{2t\widehat{\omega}}{3}+V^{2}(x_{0},0)\big)\big)^{\prime}|_{C^{\alpha,\frac{\alpha}{2}}(S^{1}\times [0,t_{0}])} \nonumber
			\\& \leq &  |w_{1}|_{C^{2,\alpha}(S^{1})}+ |w_{2}|_{C^{2,\alpha}(S^{1})}+C |w_{1}|^{2}_{C^{1,\alpha}(S^{1})}+ C|w_{2}|^{2}_{C^{1,\alpha}(S^{1})}\nonumber 
			\\& & \quad +\frac{C}{\inf_{(x_{0},t)\in S^{1}\times [0,t_{0}]}\big( \frac{2t\widehat{\omega}}{3}+V^{2}(x_{0},0)\big)}(|\widehat{\omega}|_{C^{\alpha}(S^{1})}+|V^{2}(x_{0},0)|_{C^{1,\alpha}(S^{1})})\;\;
		\end{eqnarray}
		Using that $\inf_{(x,t)\in S^{1}\times [0,t_{0}]}\big( \frac{2t\widehat{\omega}}{3}+V^{2}(x_{0},0)\big)\geq \inf_{x\in S^{1}}V^{2}(x_{0},0)>0$,  the above bound on $H_{i}$  only depends on the initial values. 	On the other hand, in terms of $u\triangleq \left| \begin{array}{c}u_{1}\\ u_{2}\end{array}\right|$ (a $\R^{2}-$valued function), the operator $\mathscr{P}$ in \eqref{equ the integral equation} is $\alpha-$admissible, so Proposition \ref{prop general abstract short time existence} yields the solvability of \eqref{equ the integral equation}. The proof of Proposition \ref{prop short time existence} is complete. 
	\end{proof}
	
	%From solution of Equations \eqref{degenerate-system}, a simple calculation shows 
	%\begin{equation}
	%\begin{split} 
	%\frac{d}{dt} &\int_{S^1} f_1 V\dd \theta =0\\
	%\frac{d}{dt}& \int_{S^1} f_2 V\dd \theta =0\\
	%\frac{d}{dt}& \int_{S^1} f_1^{-1}f_2^{-1} V\dd \theta =0
	%\end{split}
	%\end{equation}
	
	\section{Long time existence and Convergence}
	
	\subsection{A priori estimates}
	
	\subsubsection{Geometries all quasi-isometric}\label{section-quasi-isometric}
	Applying the \emph{maximum principle} to functions $f_1$, $f_2$ and $f_3$ (defined in Equation \eqref{EqnAVf}) satisfying the equations in Lemma \ref{f_ievolution-lemma}, we can derive that the maximal value of $f_i$ for each $i=1,2,3$ is non-increasing in $t$ (because the quantity $\mathcal{T}$ is nonnegative). Thus, $f_i$'s are uniformly bounded from above. On the other hand, since $f_1f_2f_3\equiv 1$ (see Equation \eqref{f_1 f_2 f_3=1}), we could deduce that they are also uniformly bounded below away from $0$, i.e. 
	
	\begin{equation}\label{f_ibound}
		\frac{1}{C}\leq f_i\leq C, \;\; i=1, 2, 3. 
	\end{equation}
	for some $C$ depending only on the initial values of $\sup_{S^1} f_1, \sup_{S^1} f_2$, and $\sup_{S^1} f_3$.\\

	We then list the only non-vanishing component of the various geometric quantities: 
	\begin{itemize}
		\item Christoffel symbols:
		\begin{equation} \label{Christoeffel-symbols}
			\Gamma_{00}^{\;\; 0} = \frac{V'}{V} \; , \;  \Gamma_{ii}^{\;\; 0} = -\frac{f_i'}{2V^2}\; , \;\; \Gamma_{i0}^{\;\; i}= \frac{f_i'}{2 f_i}\;\;\; ; \; \; i=1, 2, 3
		\end{equation}
		
		\item Riemannian curvature:
		\begin{equation}\label{Riemannian-curvature}
			R_{0i0}^{\;\;\;\;\; i}= \frac{f_i''}{2f_i} - \frac{f_i'^2}{4 f_i^2} - \frac{V' f_i'}{2V f_i} \;\;\; ; \; \; i=1, 2, 3
		\end{equation}
		
		\item Ricci curvature:
		\begin{equation} \label{Ricci-curvature}
			R_{00}= -\frac{1}{4} \sum_{i=1}^3 \frac{f_i'^2}{f_i^2}\;\; , \;\; R_{ii} = -\frac{f_i''}{2V^2} + \frac{V'}{2V^3}f_i' + \frac{f_i'^2}{4 V^2 f_i} \;\;\; ; \; \; i=1, 2, 3
		\end{equation}
		
		\item Scalar curvature:
		\begin{equation}\label{Scalar-curvature}
			R= - \frac{1}{2 V^2} \sum_{i=1}^3 \frac{f_i'^2}{f_i^2}= -\frac{1}{2}\mathcal{T}
		\end{equation}
	\end{itemize}

	Each meridian circle, i.e. the factor $S^1_{\bar x}=\{x_i =\bar x_i, i=1, 2, 3\}$ for fixed $\bar x$, is always a closed geodesic. This could be seen by either direct verification of geodesic equations (with arc length parametrization in the $S^1$ factor) using the above formulas for Christoffel symbols, or by the general fact from \emph{warped Riemannian geometry}.  Or else, let $p, q$ be two points on the meridian $S^1_{\bar x}$ and let $\gamma: [0, 1]\to \mathbb{T}^4$ be any smooth curve connecting $p, q$ , and let $\tilde \gamma$ be its projection to the meridian, then 
	$$l(\gamma)=\int_0^1\sqrt{V^2(x_0(t))x'_0(t)^2 + \sum_{i=1}^3 f_i(x_0(t))x_i'(t)^2}\dd t\geq \int_0^1V(x_0(t))|x_0'(t)|\dd t=l(\tilde \gamma).$$ 
	This shows that any geodesic in $\mathbb{T}^4$ connecting $p, q$ must have image contained in $S^1_{\bar x}$ and therefore must be one of the two meridian arcs (with arc length parametrization) aiming at $q$ in two opposite directions.  It follows that 
	\[
	d(p,q)=\frac{1}{2}\int_{S^1_{\bar x}} V(x_0)\dd x_0
	\]
	if the two meridian arcs connecting $p, q$ have the same length. 
	
	The next proposition shows that the meridian circles always have uniformly bounded length, i.e. 
	
	\begin{proposition}[(total volume estimate)]\label{total-volume-estimate}
		Let $\bar l= \frac{1}{6(2\pi)^3}\Big([\omega_1]^2+[\omega_2]^2+ [\omega_3]^2\Big)[\mathbb{T}^4]$, and $\underline l=\int_{S^1}V(x_0, 0)\dd x_0$. Suppose the hypersymplectic flow exists on $[0, t_0)$, the it holds that $l_{g(t)}(S_{\bar x}^1)= \int_{S^1} V(x_0, t)\dd x_0 $ lies between $\underline l$ and $\bar l$ for all $t\in [0, t_0)$. Meanwhile, the total volume is always bounded below by $\text{Vol}_{g(0)}(\mathbb{T}^4, g(0))$ and above by $\frac{1}{6}\Big([\omega_1]^2+[\omega_2]^2+ [\omega_3]^2\Big)[\mathbb{T}^4]$. 
	\end{proposition}         
	
	\begin{proof}It is a general fact that the volume of the corresponding Riemannian metrics along the hypersymplectic flow is uniformly bounded by topological data. The reasoning is the following simple inequality. 
		\begin{align*}
			\frac{1}{2}\sum_{i=1}^3\int_{\mathbb{T}^4}\omega_i\wedge\omega_i 
			& =\int_{\mathbb{T}^4} \tr Q\; \mu 
			\geq 3\int_{\mathbb{T}^4} \det Q \; \mu\\
			& =3\text{Vol}_{g(t)}\Big(\mathbb{T}^4, g(t)\Big)
			=3(2\pi)^3 \int_{S^1} V(x_0, t)\mathrm{d}x_0
		\end{align*}
		The lower bound is because the flow always increases the volume. We further get the length bound of the geodesic meridians by the discussion before this proposition. 
	\end{proof}  
	
	This is enough to prove the following \emph{non-collapsing} result.

	\begin{proposition}[(injectivity radius estimate, volume ratio estimate)]\label{injectivity}
		Suppose the hypersymplectic flow exists on $[0, t_0)$ with $t_0\leq \infty$. For the family of the corresponding Riemannian metrics $(\mathbb{T}^4, g(t)), \; t\in [0, t_0)$ above, if we rescale the metrics such that $|\text{Rm}|$ is bounded above by $1$, then the injectivity radius of $g(t)$ is uniformly bounded from below (In other words, $inj(\mathbb{T}^4, g(t))\geq \frac{i_0}{\sup_{\mathbb{T}^4} |\text{Rm}(g(t))|_{g(t)}^\frac{1}{2}+1}$ for some $i_0>0$ depending only on the initial data). Moreover, $g(t)$ has a uniform lower bound on the {volume ratio} $\frac{\text{Vol}_{g(t)}B_{g(t)}(p,r)}{r^4}$ for $r\leq r_0$.  \footnote{This is local $\kappa$-noncollapsing property in Perelman's sense. Fortunately, it holds even as $t\to \infty$ in our case. }
	\end{proposition}
	
	\begin{proof}
		For any fixed $t\in [0, t_0)$, using a new intrinsic coordinate $\{y, x_1, x_2, x_3\}$ where $\frac{\dd y}{\dd x_0}= V(x_0, t)$, the Riemannian manifold  $\big(\mathbb{T}^4, g(t)\big)$ is written as the multiply warped product $S^1\times_{\tilde f_1} S^1\times_{\tilde f_2}S^1\times_{\tilde f_3} S^1$, i.e. 
		\begin{equation} \label{first-warped-metric}
			g(t) =V^2(x_0, t)\dd x_0^2 + \sum_{i=1}^3 f_i(x_0, t)\dd x_i^2
			= \dd y^2 + \sum_{i=1}^3 \tilde f_i(y, t)\dd x_i^2
		\end{equation}  
		for $\tilde f_i (y, t) = f_i (x_0(y), t)$. Let $p\in \mathbb{T}^4$ have coordinate $(\bar y, \bar x_1, \bar x_2, \bar x_3)$, and for any $r\leq \min \{\frac{1}{2} \underline l, \pi\}$ define a coordinate chart
		\[
		\Omega_r=\{(y, x_1, x_2, x_3)\in \mathbb{R}^4|\; \left|y-\bar y\right| < r, \left|x_i -\bar x_i\right| < r, i=1, 2, 3\}.
		\]
		It is clearly that the image of the coordinate chart $\Omega_r$ in $\mathbb{T}^4$ is contained in $B_{g(t)} (p, 2\sqrt{C}r)
		$ since $\tilde f_i \leq C$ ($C>1$ here). As a consequence, $\forall\; r\leq r_0=\frac{1}{2\sqrt{C}}\min\{\frac{1}{2}\underline l, \pi\}$, 
		
		\begin{equation}\label{volume-ratio-lower-bound}
			\begin{split} 
				\frac{\text{Vol}_{g(t)} B_{g(t)}(p, r)}{r^4}
				& \geq \frac{\text{Vol}_{g(t)} \Omega_{\frac{r}{2\sqrt{C}}}}{r^4}\\
				& =\frac{1}{r^4} \int_{\Omega_{\frac{r}{2\sqrt{C}}}} \tilde f_1 \tilde f_2 \tilde f_3\dd y\dd x_1\dd x_2\dd x_3\\
				& = \frac{1}{C^2}
			\end{split}
		\end{equation}

		If the family $g(t)$ has uniform bounded $|\text{Rm}(g(t))|_{g(t)}$, then we could normalize it to be bounded by $1$. For this normalized family, the \emph{volume ratio} is also uniformly bounded from below (because the volume ratio is scaling invariant). By Cheeger-Gromov-Taylor's estimate \cite[Theorem 4.7]{CGT} on the injectivity radius of a Riemannian metric in terms of volume ratio and curvature bound, the family has uniform lower bound on the injectivity radius. If the family does not have a uniform bound on $|\text{Rm}(g(t))|_{g(t)}$, we blow up the family (rescaling $\tilde g(t)=(\sup_{\mathbb{T}^4} |\text{Rm}(g(t))|_{g(t)} ) g(t)$) and keep the uniform lower bound on the \emph{volume ratio} (for all radius $r\leq r_0$ by Equation \eqref{volume-ratio-lower-bound}) because of the scaling invariance of the volume ratio, more precisely $\frac{\text{Vol}_{\lambda^2 g}B_{\lambda^2 g}(p,\lambda r)}{(\lambda r)^n}=\frac{\text{Vol}_g B_g(p, r)}{r^n}$. The same reason as the previous case show us the rescaled family has uniform lower bound on injectivity radius.
	\end{proof}
	
	This subsection shows that the Riemannian metrics along the flow are all equivalent to the standard Euclidean metric on $\mathbb{T}^4$ modulo diffeomorphisms. 
	
	\subsubsection{Scalar curvature increasing}\label{section-scalar-curvature-increasing}
	From Equations (\ref{Tformula}), (\ref{EqnV}) and (\ref{evolution-log-f_i}) we can derive the evolution equation of $\mathcal{T}$ as the following:
	
	\begin{equation}\label{T^2evolution}
		\begin{split}
			\partial_t \mathcal{T} 
			& =
			-\frac{2}{3}\mathcal{T}^2
			+
			\frac{2}{V^2} \sum_{i=1}^3 (\log f_i)' (\frac{\partial }{\partial t}\log f_i)' \\
			& =
			-\frac{2}{3}\mathcal{T}^2
			+
			\frac{2}{V^2} \sum_{i=1}^3 (\log f_i)' \Big\{\frac{1}{V^2} (\log f_i)'' + \frac{1}{V^2} \Big( (\log f_i)' \Big)^2 - \frac{V'}{V^3} (\log f_i)' - \frac{1}{3}\mathcal{T}\Big\}' \\
			& =
			-\frac{2}{3}\mathcal{T}^2
			+
			\frac{2}{V^2} \sum_{i=1}^3 (\log f_i)' \Big\{ \frac{1}{V^2} \Big( (\log f_i)''' +  2(\log f_i)'(\log f_i)''\Big)\\
			& \quad \qquad \qquad  - \frac{2V'}{V^3} \Big((\log f_i)'' + \big( (\log f_i)' \big)^2\Big)
			- (\frac{V'}{V^3})'(\log f_i)' - \frac{V'}{V^3}(\log f_i)'' - \frac{1}{3}(\mathcal{T})' \Big\}\\
			& =
			-\frac{2}{3}\mathcal{T}^2
			+ \frac{2}{V^4} \sum_{i=1}^3 \Big\{ (\log f_i)' (\log f_i)''' +  2\Big( (\log f_i)' \Big)^2(\log f_i)''
			- \frac{3V'}{V}(\log f_i)' (\log f_i)'' \\
			& \quad \qquad\qquad -\frac{2V'}{V} \Big( (\log f_i)' \Big)^3\Big\} 
			- (\frac{V'}{V})'\frac{2\mathcal{T}}{V^2} + 4 (\frac{V'}{V})^2\frac{\mathcal{T}}{V^2}\\
			& =
			-\frac{2}{3}\mathcal{T}^2
			+ \frac{2}{V^4} \Big\{ \sum_{i=1}^3 (\log f_i)' (\log f_i)''' - (\frac{V'}{V})'\mathcal{T}V^2\Big\}
			+  \frac{4}{V^4}\sum_{i=1}^3 \Big( (\log f_i)' \Big)^2(\log f_i)''\\
			& \quad \qquad\qquad - \frac{6V'}{V^5}\sum_{i=1}^3 (\log f_i)' (\log f_i)''  -\frac{4V'}{V^5}\sum_{i=1}^3 \Big((\log f_i)'\Big)^3
			+ 4 (\frac{V'}{V})^2\frac{\mathcal{T}}{V^2}
		\end{split}
	\end{equation}
	
	Observe that there is a big negative term $-\frac{2}{3}\mathcal{T}^2$ in the evolution of $\mathcal{T}$, which is good for our control on $\mathcal{T}$. However, this is followed by a complicated terms whose sign is not clear. Fortunately, we are able to show that at the maximum point of $\mathcal{T}$, the complicated term is not bigger than $\frac{1}{3}\mathcal{T}^2$ (see the lemma below). 
	
	Denote $\overline{\mathcal{T}}(t)=\max_{x\in S^1} \mathcal{T}(x, t)$. The next Lemma is a differential inequality about $\overline{\mathcal{T}}$, which in particular implies it is decreasing. 
	\begin{lemma}\label{lemma-T-decreasing}
		\begin{equation}
			\frac{\dd }{\dd t}\overline{\mathcal{T}}\leq 
			-\frac{1}{3}\overline{\mathcal{T}}^2
		\end{equation}
	\end{lemma}
	\begin{proof}
		By the defining formula \eqref{Tformula}, the first and second derivatives of $\mathcal{T}$ are: 
		
		\begin{equation} 
			\begin{split} 
				\mathcal{T}' 
				&= \frac{2}{V^2} \sum_{i=1}^3 (\log f_i)' (\log f_i)''- \frac{2V'}{V}
				\mathcal{T}\\
				\mathcal{T} ''
				& = \frac{2}{V^2} \sum_{i=1}^3 (\log f_i)' (\log f_i)''' + \frac{2}{V^2}\sum_{i=1}^3 \Big((\log f_i)''\Big)^2 - \frac{4V'}{V^3}\sum_{i=1}^3 (\log f_i)' (\log f_i)'' \\
				& \quad\qquad -2 (\frac{V'}{V})' \mathcal{T} - \frac{2V'}{V} \mathcal{T} '
			\end{split}
		\end{equation}
		At the maximum point $(p,t)\in S^1\times [0, t_0)$ of $\mathcal{T}$ (restricted to $S^1\times \{t\}$) we have 
		\[
		\mathcal{T}'=0\;\;, \;\; \mathcal{T}''\leq 0
		\]
		and therefore it holds that
		\begin{equation}\label{FirstDerivative}
			\frac{V'}{V}
			=\frac{1}{\overline{\mathcal{T}}V^2}\sum_{i=1}^3 (\log f_i)'(\log f_i)''
		\end{equation}
		and
		\begin{equation} \label{SecondDerivative}
			\begin{split}
				\sum_{i=1}^3 & (\log f_i)'(\log f_i)''' - (\frac{V'}{V})'\overline{\mathcal{T}}V^2 \\
				&  \leq -\sum_{i=1}^3 \Big((\log f_i)''\Big)^2 + \frac{2V'}{V} \sum_{i=1}^3 (\log f_i)'(\log f_i)''\\
				& = -\sum_{i=1}^3 \Big((\log f_i)''\Big)^2 + \frac{2}{\overline{\mathcal{T}}V^2} \Big\{\sum_{i=1}^3 (\log f_i)'(\log f_i)''\Big\}^2
			\end{split}
		\end{equation}

		Substituting the inequality \eqref{SecondDerivative} into the term $\left\{, \right\}$ in Equation \eqref{T^2evolution} and replacing $\frac{V'}{V}$ in \eqref{T^2evolution} by the RHS of Equation \eqref{FirstDerivative}, we get the differential inequality for $\overline{\mathcal{T}}$ as:

		\begin{equation}\label{T^2estimate}
			\begin{split}
				\frac{\dd}{\dd t}\overline{\mathcal{T}}  
				\leq &  -\frac{2}{3}\overline{\mathcal{T}}^2 - \frac{2}{V^4}\sum_{i=1}^3 \Big((\log f_i)''\Big)^2 
				+ \frac{2}{\overline{\mathcal{T}} V^6}\Big\{\sum_{i=1}^3 (\log f_i)'(\log f_i)''\Big\}^2 \\
				& + \frac{4}{V^4}\sum_{i=1}^3 \Big((\log f_i)'\Big)^2(\log f_i)''
				- \frac{4}{\overline{\mathcal{T}} V^6} \sum_{i=1}^3 \Big((\log f_i)'\Big)^3\sum_{j=1}^3 (\log f_j)'(\log f_j)''\\
				= & -\frac{2}{3}\overline{\mathcal{T}}^2
				- \frac{2}{\overline{\mathcal{T}}V^6}\Big\{\overline{\mathcal{T}}V^2\sum_{i=1}^3 \Big((\log f_i)''\Big)^2 
				- \Big(\sum_{i=1}^3 (\log f_i)'(\log f_i)''\Big)^2 \\
				& - 2 \overline{\mathcal{T}}V^2\sum_{i=1}^3 \Big((\log f_i)'\Big)^2(\log f_i)''
				+ 2\sum_{i=1}^3 \Big((\log f_i)'\Big)^3\sum_{j=1}^3 (\log f_j)'(\log f_j)''\Big\}
			\end{split}
		\end{equation}
		To simplify the notation, denote the values at $p$ of the various functions:
		\begin{align*} 
			(\log f_1)' & =a, (\log f_2)'  =b, (\log f_3)'   =c= -(a+b)\\
			(\log f_1)''& =A, (\log f_2)'' =B,  (\log f_3)''  =C= -(A+B)
		\end{align*}
		Then the term in $\Big\{,\Big\}$ of the RHS in the above inequality (\ref{T^2estimate}) is 
		\begin{equation}
			\begin{split}
				(a^2+  b^2& +c^2)(A^2+B^2+C^2) 
				- (aA+bB+cC)^2\\
				& \qquad - 2(a^2+b^2+c^2)(a^2A+b^2B+c^2C) + 2(a^3+b^3+c^3)(aA+bB+cC)\\
				& 
				= 3 (bA-aB)^2 - 2 (a-b)(2a+b)(a+2b)(bA-aB)\\
				& 
				= 3 \left\{ (bA-aB) - \frac{1}{3}(a-b)(2a+b)(a+2b)\right\}^2 - \frac{1}{3}(a-b)^2(2a+b)^2(a+2b)^2
			\end{split}
		\end{equation} 
		Inserting it into the inequality (\ref{T^2estimate}) gives
		
		\begin{equation}
			\begin{split}
				\frac{\dd}{\dd t}\overline{\mathcal{T}}
				&\leq -\frac{2}{3}\overline{\mathcal{T}}^2 
				- \frac{6}{\overline{\mathcal{T}}V^6} \left\{ (bA-aB) - \frac{1}{3}(a-b)(2a+b)(a+2b)\right\}^2\\
				& \qquad\qquad\qquad\qquad\quad\quad+ \frac{2}{3\overline{\mathcal{T}}V^6}(a-b)^2(2a+b)^2(a+2b)^2\\
				& \leq -\frac{2}{3\overline{\mathcal{T}}V^6}\left\{ (a^2+b^2+c^2)^3 - (a-b)^2(2a+b)^2(a+2b)^2\right\}\\
				&=  -\frac{2}{3}\frac{8 (a^2+b^2 +ab)^3 - (a-b)^2(2a+b)^2(a+2b)^2}{8 (a^2+ab+b^2)^3}\overline{\mathcal{T}}^2\\
				& =_{x=\frac{a}{b}} -\frac{2}{3} \left\{ \frac{1}{2}+ \frac{27}{8(x^2+x+1)}(1-\frac{1}{x^2+x+1})^2\right\}\overline{\mathcal{T}}^2\\
				& \leq -\frac{1}{3}\overline{\mathcal{T}}^2
			\end{split}
		\end{equation}
	\end{proof}
	
	The Lemma \ref{lemma-T-decreasing} implies the following \emph{a priori estimate} about torsion tensor (or equivalently, the scalar curvature by Formula \eqref{Scalar-curvature}):
	
	\begin{proposition}[{torsion-scalar curvature estimate}]\label{torsion-tensor-estimate}
		Suppose the hypersymplectic flow of simple type on $\mathbb{T}^4$ exists on $[0, t_0)$, then the torsion tensor is uniformly bounded on $[0, t_0)$. More precisely, let $\mathcal{T}_0=\max_{p\in \mathbb{T}^4} \mathcal{T}(p, 0)$, then the following decaying estimate holds:  
		\[
		\max_{p\in \mathbb{T}^4} \mathcal{T}(p, t)\leq 
		\frac{\mathcal{T}_0}{1+\frac{1}{3}\mathcal{T}_0 t}, \;\; \forall t\in [0, t_0). 
		\]
		The minimum of the scalar curvature is increasing, and the maximum of the scalar curvature is always non-positive. 
	\end{proposition}
	It is interesting to note that Lauret obtained similar bound for homogeneous $G_2$-Laplacian flow solution \cite[Proposition 5.21]{L1}. 
	
	%%%%%%%%%%%%%%%%%%%%%%%%%%%%%%%%%%%%
	%%%%%%%%%%%%%%%%%%%%%%%%%%%%%%%%%%%%
	%%%%%%%%%%%%%%%%%%%%%%%%%%%%%%%%%%%%
	
	\subsection{Long time existence}\label{section-long-time-existence}
	\begin{theorem}[(Long time existence)]\label{long-time-existence}
		Initiated from any smooth hypersymplectic structure of simple type on $\mathbb{T}^4$, the hypersymplectic flow
		exists on $[0, \infty)$.
	\end{theorem}
	
	\begin{proof}
		First, we show that if the flow (\ref{definite-triple-flow-of-simple-type}) exists on $[0, t_0) $ with $t_0<\infty$ and $\mathcal{T}$ is uniformly bounded, then the flow could be extended across $t_0$.
		
		Proposition \ref{injectivity} shows that the family of metrics $g(t)$ is uniform noncollapsing on any scale smaller than the fixed scale $r_0$. Suppose the flow is not extendible across $t_0$, we get that the uniform uppper bound of $\Lambda_{\phi} = \sup_{\mathbb{T}^4\times \mathbb{T}^3}(|\text{Rm}|_{g_\phi}^2+|\nabla T|_{g_\phi}^2)^\frac{1}{2}$ for $g_{\phi(t)} = g(t)\oplus g_Q=g(t)\oplus Q_{ij}\dd t_i\dd t_j$ on 
		$\mathbb{T}^7=\mathbb{T}^4\times \mathbb{T}^3$ does not exist for $t\in [0, t_0)$ by \cite[Theorem 5.1]{LW1} (where the anti-symmetric $2$-tensor $T$ is $-2$ mutiple of the intrinsic torsion $2$-form $\bm{\tau}$ in Equation \eqref{intrinsic-torsion}). Then there exists a sequence $g^i=g(t_i)$ with $t_i\to t_0$ such that
		\[
		\Lambda^i = \Lambda_{{\phi(t_i)}} = \sup_{\mathbb{T}^4\times \mathbb{T}^3}(|\text{Rm}|_{g_\phi}^2+|\nabla T|_{g_\phi}^2)^\frac{1}{2}(\cdot, t_i)\to \infty,
		\]
		and suppose this upper bound is attained at $p_i$. 
		Let $\tilde g^i =\Lambda^i g^i$ and $\tilde{\underline{\omega}}^i=\Lambda^i \underline\omega^i$ and similarly $\tilde\phi^i=\Lambda^i \phi(t_i)$, then $\tilde g^i$ has uniform lower bound on the injectivity radius by Proposition \ref{injectivity} and uniform lower bound 
		on the volume ratio on all scales, and moreover $\Lambda_{\tilde\phi^i}$ is uniformly bounded. The same argument as in the proof of \cite[Theorem 5.1]{FY} shows that
		we can take a Cheeger-Gromov limit 
		\[
		(\mathbb{T}^4, \tilde g^i , \tilde \omega^i)\longrightarrow (X^\infty, \tilde g^\infty, \tilde \omega^\infty), 
		\]
		where $\tilde\omega^\infty$ and $\tilde g^\infty$ defines a complete hyperK\"ahler structure on $X^\infty$. 
		Let $\gamma_i$ be the meridian geodesic segment connecting $q_i$ and $q_i'$ whose length is \emph{half}  of the length of the meridian circle they lie on such that $p_i$ is the distance midpoint of $\gamma_i$. We know from Proposition \ref{total-volume-estimate} that $\gamma_i$ is a minimizing geodesic for $g^i$, and its length is bounded from below. Under the rescaled metric $\tilde g^i$, this gives a sequence of minimizing geodesic $\tilde \gamma_i$ whose length tends to infinity. The limit $\tilde\gamma_\infty$ will become a geodesic line on 
		$(X^\infty, \tilde g^\infty)$.  
		
		By the \emph{Cheeger-Gromoll splitting theorem}, we know that $(X^\infty, \tilde g^\infty)$ must be flat. By the smooth convergence, we have 
		\[
		|\text{Rm}({g^\infty})|_{g^\infty}(p_\infty)=1,
		\]
		a contradiction.

		By Proposition \ref{torsion-tensor-estimate}, we know $\mathcal{T}$ is uniformly bounded at any finite time, thus the theorem is proved.
	\end{proof}

	\subsection{Convergence}\label{section-convergence}
	\begin{theorem}[(Convergence)]\label{convergence}
		Initiated from any smooth hypersymplectic structure of simple type on $\mathbb{T}^4$, the hypersymplectic flow existing on $[0, \infty)$
		converges smoothly to the standard hyperK\"ahler structure  $\underline\omega^0=(\omega^0_1, \omega^0_2, \omega^0_3)$ on $\mathbb{T}^4$ modulo diffeomorphisms.  More precisely, there exists a family of orientation preserving diffeomorphisms $F_t$ of $\mathbb{T}^4$, all fixing the $\mathbb{T}^3$ factor of $\mathbb{T}^4=S^1\times \mathbb{T}^3$ pointwisely , such that 
		\[
		F_t^*\underline{\omega}(t)\xrightarrow{C^\infty} _{t\to \infty}\underline{\omega}^0
		\]
	\end{theorem}
	
	\begin{proof}
		We use the fact that $\mathcal{T}$ is uniformly bounded on $[0, \infty)$. Because Proposition \ref{injectivity} hold independently of the maximal existing time $t_0$,  
		exactly the same proof as the above \emph{long time existence} shows that 
		for the closed $G_2$-structure $\phi(t)=\dd t^1\wedge \dd t^2\wedge\dd t^3 - \dd t^1\wedge\omega_1(t) - \dd t^2\wedge\omega_2(t) - \dd t^3\wedge\omega_3(t)$ on $\mathbb{T}^7$, 
		\[
		\sup_{\mathbb{T}^4\times [0, \infty)} \Big( |\text{Rm}|_{g}^2 + |\nabla T|_{g}^2\Big)\leq C_0
		\]
		for some uniform constant $C_0>0$. Thus Proposition \ref{injectivity} gives us a uniform lower bound on the injectivity radius. By \emph{Shi-type estimate} \cite{LW1}, there exists $C_k>0$  independent of $t\in [0, \infty)$ for all $k\in \mathbb{N}$ such that
		\[
		|\nabla^k \text{Rm}|_{g}^2 + |\nabla^{k+1} T|_{g}^2\leq C_k
		\]
		for any $k\geq 1$. 
		
		These obtained bounds do not contain information of higher derivatives of $f_i$ and $V$ separatedly, and thus it is hard to conclude the convergence of these functions. We are going to use diffeomorphisms to pull back the hypersymplectic structures such that $V$ becomes $1$ and the derivatives bound of Riemannian curvatures become the derivatives bound of the new warped functions. 
		
		Let $v_t=\int_{S^1}V(\theta, t)\dd \theta$, then $v_t$ is increasing according to $t$ since the total volume is increasing, moreover it is bounded from above by Proposition \ref{total-volume-estimate}. Denote the limit constant by $v_\infty$. Viewing $V(\cdot, t)$ as a periodic function (with period $2\pi$) naturally,  we could define a diffeomorphism 
		\begin{equation*} 
			\begin{split} 
				G_t: \mathbb{R} & \longrightarrow\mathbb{R}\\
				x_0 & \mapsto \frac{2\pi}{v_t}\int_0^{x_0} V(x', t)\dd x'
			\end{split}
		\end{equation*}
		This map is an orientation-preserving diffeomorphism because $V$ is positive and smooth. It is periodic with period $2\pi$ and thus descends to a diffeomorphism (still denoted by): 
		\begin{equation*} 
			\begin{split} 
				G_t: S^1=\mathbb{R}/2\pi \mathbb{Z} & \longrightarrow S^1= \mathbb{R}/2\pi \mathbb{Z}\\
				x_0\;\; (\mod 2\pi) & \mapsto G_t(x_0) \;\; (\mod 2\pi)
			\end{split}
		\end{equation*}
		And this extends to an orientation-preserving diffeomorphism (still denoted by): 
		\begin{equation}
			\begin{split} 
				G_t : \mathbb{T}^4 & \longrightarrow \mathbb{T}^4\\
				(\theta, x_1, x_2, x_3) &\mapsto (G_t(\theta), x_1, x_2, x_3)
			\end{split}
		\end{equation}
		We write the new function
		\begin{equation*} 
			\begin{split} 
				\hat f_i\; : \;\; S^1\times [0, t_0) 
				& \longrightarrow \mathbb{R}^+\\
				(\theta, t)  
				& \mapsto f_i(G_t^{-1}(\theta), t)
			\end{split}
		\end{equation*}
		Because $\omega_i(t)=f_i(x_0, t) V(x_0, t)\dd x_0\wedge \dd x_i + \frac{1}{2}\epsilon_{ijk}\dd x_j\wedge \dd x_k$, we have the push-forward hypersymplectic structure
		\begin{equation}
			\begin{split} 
				\widehat{\omega}_i(t) 
				& := (G_t^{-1})^*\omega_i(t) \\
				& = f_i(G_t^{-1}(y), t)\frac{v_t}{2\pi}\dd y\wedge \dd x_i + \frac{1}{2}\epsilon_{ijk}\dd x_j\wedge\dd x_k \\
				& =\frac{v_t}{2\pi} \hat f_i(y, t)\dd y\wedge \dd x_i + \frac{1}{2}\epsilon_{ijk}\dd x_j\wedge\dd x_k 
				%   & \xrightarrow{C^\infty} 
				%      \frac{v_\infty}{2\pi}c_i \dd y\wedge \dd x_i 
				%      + \frac{1}{2}\epsilon_{ijk}\dd x_j\wedge\dd x_k 
			\end{split}
		\end{equation}
		The corresponding Riemannian metric is $\big(\mathbb{T}^4, \widehat g(t)=(G_t^{-1})^*g(t) \big)= S^1\times_{\hat f_1} S^1\times_{\hat f_2} S^1 \times_{\hat f_3} S^1$. The previous bounds we obtained are in nicer forms in the $y$-coordinate, under which we are going to show that the flow converges. Concretely, we have
		\begin{equation}\label{variousbounds}
			\begin{split} 
				& \frac{1}{C} \leq \hat f_i \leq C\\ 
				&\widehat{\mathcal{T}}
				=\sum_{i=1}^3 \frac{1}{\hat f_i^2} (\frac{\partial \hat f_i}{\partial y})^2 \xrightarrow{C^0}_{t\to \infty} 0\\
				& |\widehat{\nabla}^k \widehat{\Rm}|_{\hat g}^2 
				\leq C_k,\;\; \forall\;\; k=0, 1, 2, \cdots
			\end{split}
		\end{equation}
		
		According to Equation \eqref{Riemannian-curvature} of Riemannian curvature for the metric $\widehat{g}(t)$,
		\[
		|\widehat{\Rm}|_{\hat g}^2
		= \sum_{i=1}^3 \left|\frac{1}{2\hat f_i} \frac{\partial^2 \hat f_i}{\partial y^2} - \frac{1}{4 \hat f_i^2} (\frac{\partial \hat f_i}{\partial y})^2\right|^2
		\]
		and therefore the formulas in \eqref{variousbounds} imply $\left|\frac{\partial^2\hat f_i}{\partial y^2}\right|$ is uniformly bounded. Similarly, for $k\in \mathbb{N}^+$, $|\widehat{\nabla}^k\widehat{\Rm}|_{\hat g}^2$ is the sum of
		$\left|\frac{1}{2\hat f_i} \frac{\partial^{k+2} \hat f_i}{\partial y^{k+2}}\right|^2$ and terms involving $\hat f_i$ and lower order derivatives of $\hat f_i$. Inductively, we know that the $k$-th ($k\geq 1$) derivative of $\hat f_i$ with respect to $y$ is uniformly bounded on $\mathbb{T}^4\times [0, \infty)$. Since $\hat f_i$ is uniformly bounded in $C^{k+1}$ given any fixed $k\in \mathbb{N}$, the \emph{Arzel\`a-Ascoli theorem} implies that for any sequence $t_\alpha\to \infty$, there exists a subsequence $t_{\alpha_m}$ such that $\hat f_{i}(\cdot, t_{\alpha_m})$ converges in $C^k$. Moreover, because the first derivative of $\hat f_i$ with respect to $y$ tends to $0$ as $t\to \infty$ (the second equation in \eqref{variousbounds}), the limits must all be constants, i.e. there exists three constants $c_1, c_2, c_3>0$ such that 
		\begin{equation} 
			\begin{split}
				\widehat{\omega}_i(t_{\alpha_m})\xrightarrow{C^k}_{m\to \infty} 
				\frac{v_\infty}{2\pi}c_i \dd y\wedge \dd x_i + \frac{1}{2}\epsilon_{ijk}\dd x_j\wedge\dd x_k\;,  \;\; i=1, 2, 3. 
			\end{split}
		\end{equation}
		
		On the other hand, since for each $i$, 
		\begin{equation}
			\begin{split} 
				\int_{\mathbb{T}^4}  \widehat\omega_i(t_{\alpha_m}) \wedge       \widehat\omega_i(t_{\alpha_m}) 
				& = \int_{\mathbb{T}^4} \omega_i(t_{\alpha_m}) \wedge \omega_i(t_{\alpha_m}) 
				= [\omega_i(0)]^2[\mathbb{T}^4]\\
				& = \int_{\mathbb{T}^4} \Big(\frac{v_\infty}{2\pi}c_i \dd y\wedge \dd x_i 
				+ \frac{1}{2}\epsilon_{ijk}\dd x_j\wedge\dd x_k \Big)^2\\
				& = \frac{v_\infty}{2\pi} c_i [\omega_i(0)]^2[\mathbb{T}^4]
			\end{split}
		\end{equation}
		We conclude that $\frac{v_\infty}{2\pi}c_i =1$ for $ i=1, 2,3 $. The limit symplectic $2$-form is thus $\omega_i^0$ and independent of the particular sequence $t_\alpha$. Because of the existence and uniqueness of all possible subsequential limits of the family $\{\widehat{\underline\omega}(t)\}_{t\in [0,\infty)}$, it actually holds that
		\begin{equation}
			\widehat{\underline{\omega}}(t)= (G_t^{-1})^*\underline{\omega}(t) 
			\xrightarrow{C^\infty}_{t\to \infty} \underline{\omega}^0
		\end{equation}
		where ``$C^\infty$'' means ``$C^k$'' for any $k\in \mathbb{N}$. The convergence statement in this theorem is established by setting $F_t=G_t^{-1}$.
	\end{proof}
	
	The family of closed $G_2$-structures $(\mathbb{T}^7, \phi(t))_{t\in [0, \infty)}$ where
	\[
	\phi(t)=\dd t^1\wedge\dd t^2\wedge\dd t^3 - \dd t^1\wedge\omega_1(t) - \dd t^2\wedge \omega_2(t) - \dd t^3\wedge\omega_3(t)
	\]
	is a family of cohomogeneity-one closed $G_2$-structures satisfying the $G_2$-Laplacian flow, and $F_t^*\phi(t)$
	converges smoothly to the standard torsion free $G_2$-structure, where each $F_t: \mathbb{T}^7=S^1\times \mathbb{T}^6\longrightarrow \mathbb{T}^7=S^1\times \mathbb{T}^6$ is a diffeomorphism only reparameterizing the $S^1$ factor.
	
\section*{Acknowledgement}
		Part of this work was done in the spring semester of 2016 when CY was a Viterbi Postdoctoral Fellow at MSRI and HH was visiting MSRI.  Both authors are very grateful to MSRI for providing such a wonderful semester-long program in differential geometry in a splendid and friendly environment. We would like to thank professor Joel Fine and Jason Lotay for many stimulating discussions.  We would also like to thank professor Simon Donaldson for his interest in this work. The suggestions from the referee are greatly appreciated.

\AtEndDocument{\bigskip{\footnotesize%
		\textsc{Hongnian Huang: Department of Mathematics and Statistics, University of New Mexico, Albuquerque, NM, 87131, U.S.A.}\par
		\textit{E-mail address} : \texttt{hnhuang@unm.edu}\\
		
		\textsc{Yuanqi Wang: Simons Center for Geometry and Physics, Stony Brook University, Stony Brook, NY, 11794, U.S.A.} \par  
		\textit{E-mail address} : \texttt{ywang@scgp.stonybrook.edu}\\
		
		\textsc{Chengjian Yao: D\'epartement de Math\'ematique, Universit\'e libre de Bruxelles(ULB), CP 218, Boulevard du Triomphe, B-1050 Bruxelles, Belgium} \par  
		\textit{E-mail address} : \texttt{Chengjian.Yao@ulb.ac.be}
}}

\begin{thebibliography}{99}% Replace 9 by 99 if 10 or more references
		%
		% Please note the use of "\and" between author names below
		%
		\bibitem{B}{ R. Bryant}, `Some remarks on $G_2$ structures'. {\em Proceedings of G\"okova Geometry-Topology Conference}, 2005, 75-109.
		%
		\bibitem{BX}{ R. Bryant \and F. Xu}, `Laplacian flow for closed $G_2$ structure: short time behavior'.  {arXiv:1101.2004}.
		
		\bibitem{CGT}J. Cheeger, M. Gromov, and M. Taylor, `Finite propagation speed, kernel estimates for functions of the Laplace operator, and the geometry of complete Riemannian manifolds'. J. Differential Geom. 17 (1982), no. 1, 15--53. doi:10.4310/jdg/1214436699. 
		%
		\bibitem{D1}{ S. Donaldson}, `Two-forms on four-manifolds and elliptic equations. Inspired by S.S. Chern', {\em 153-172, Nankai Tracts Math.} 11, World Scientific, Hackensack NJ, 2006.
		%
		\bibitem{D2}{ S. Donaldson}, `Boundary value problems in dimensions seven, four and three related to exceptional holonomy', arXiv:1708.01649.
		%
		\bibitem{FFM}{ M.  Fern\'endez, A.  Fino \and V. Manero}, `Laplacian flow for closed $G_2$ structures inducing nilsolitons', {\em The J. Geom. Anal.}, July 2016, Volume 26, Issue 3, pp. 1808-1837. 
		%
		\bibitem{FR}{ A. Fino \and A. Raffero}, `Closed warped $G_2$-structures evolving under the Laplacian flow', arXiv:1708.00222.
		%
		\bibitem{FY}{ J. Fine \and C.-J. Yao}, `Hypersymplectic 4-manifolds, the $G_2$-Laplacian flow and extension assuming bounded scalar curvature', arXiv:1704.07620.
		%
		\bibitem{GT}{ D. Gilbarg \and N.S. Trudinger}, `Elliptic Partial Differential Equations of Second Order'. {\em Grundlehren}, Vol. 224, Springer-Verlag, Berlin, 1983.
		%
		\bibitem{H}{ N. Hitchin}, `The geometry of three-forms in six dimensions'. {\em J. Diff. Geom.} 55 (2000), 547-576.
		%
		\bibitem{L1}{ J. Lauret}, `Laplacian flow for homogeneous $G_2$-structures and its solitons'. {\em Proc. Lond. Math. Soc. (3)}, 114(3), 524-570, 2017.
		%
		\bibitem{L2}{ J. Lauret}, `Laplacian solitons: questions and examples', {\emph Differential Geometry and its Applications.} 54(2017) 345-360.
		%
		\bibitem{LSU}{ O. A. Lady$\check{\text{z}}$enskaja, V. A. Solonnikov \and N. N. Ural'ceva}, `Linear and quasi-linear equations of parabolic type'.  {\em Translations of Mathematical Monographs 23, Providence, RI: American Mathematical Society}.
		%
		\bibitem{LW1}{ J. Lotay \and Y. Wei}, `Laplacian flow for closed $G_2$ structures: Shi-type estimates, uniqueness and compactness',  {\em Geom. Funct. Anal.} 27 (2017), 165-233.
		%
		\bibitem{LW2}{ J. Lotay \and Y. Wei}, `Stability of Torsion-free $G_2$ structures along the Laplacian flow', to appear in J. Diff. Geom, arXiv:1504.07771.
		%
		\bibitem{N}{ M. Nicolini}, `Laplacian solitons on nilpotent Lie groups', arXiv:1608.08599, to appear in the Bulletin of the Belgian Mathematical Society.
		%
		\bibitem{Sch}{ W. Schlag},  `Schauder and $L^{p}$ estimates for parabolic systems via campanato spaces',  { \em Communications in Partial Differential Equations.} Volume 21, 1996 - Issue 7-8.
	\end{thebibliography}
\end{document}